\documentclass[11pt]{amsart}
\usepackage{amsrefs}
\usepackage{amsfonts, amsmath, amssymb, amscd, amsthm, bm}
\usepackage{cases}
\usepackage{enumerate}
\allowdisplaybreaks[4]
\usepackage{multicol}
\usepackage{comment}

\usepackage[colorlinks,
linkcolor=blue,
anchorcolor=black,
citecolor=red
]{hyperref}

\newtheorem{thm}{Theorem}[section]
\newtheorem{cor}[thm]{Corollary}
\newtheorem{con}[thm]{Conjecture}
\newtheorem{lem}[thm]{Lemma}
\newtheorem{prop}[thm]{Proposition}
\theoremstyle{definition}
\newtheorem{defn}[thm]{Definition}
\newtheorem{rem}[thm]{Remark}
\newtheorem{ex}[thm]{Example}

\numberwithin{equation}{section}
\def\R{\mathbb{R}}

\def\et{\tilde{e}}
\def\th{\tilde{h}}
\def\tH{\tilde{H}}

\def\tg{\tilde{g}} 
\def\tgm{\tilde{g}_M} 
\def\tgs{\tilde{g}_{\Sigma}} 

\def\cm{\nabla^{M}} 
\def\cs{\nabla^{\Sigma}} 
\def\tcm{\nabla^{\tilde{M}}} 
\def\tcs{\nabla^{\tilde{\Sigma}}} 

\def\nc{\nabla^{\bot}} 
\def\tnc{\tilde{\nabla}^{\bot}} 

\def\tom{\tilde{\omega}}
\def\tt{\tilde{\theta}}

\def\tb{\tilde{b}}

\def\tu{\tilde{u}}
\def\tV{\tilde{V}}

\def\sff{\mathrm{II}}

\def\Joc{\mathfrak{J}}

\def\l{\langle}
\def\r{\rangle}

\def\vol{\mathrm{vol}}

\DeclareMathOperator\tr{tr}

\DeclareMathOperator{\cod}{codim}

\begin{document}
\title[A conformal invariant and its application]{A conformal invariant and its application to the nonexistence of minimal submanifolds}

\author{Hang Chen}
\address[Hang Chen]{School of Mathematics and Statistics, Northwestern Polytechnical University, Xi' an 710129, P. R. China \\ email: chenhang86@nwpu.edu.cn}
\thanks{Chen supported by Shaanxi Fundamental Science Research Project for Mathematics and Physics (Grant No.~22JSQ005)}

\begin{abstract}
	Let $(M^m,g)$ be an $m$-dimensional closed Riemannian manifold with non-negative sectional curvatures, $m\ge 3$. We define a conformal invariant and prove that, if the conformal invariant is bounded from above by a constant depending only on $m$, then there are no closed $n$-dimensional stable minimal submanifolds in $M$ for all $\xi(m)\le n\le m-2$, where $\xi(m)=1$ when $3\le m\le 5$ and $\xi(m)=2$ when $m\ge 6$.
	In particular, a conformal $m$-sphere with non-negative sectional curvatures does not admit any closed $n$-dimensional stable minimal submanifold for all $\xi(m)\le n\le m-2$.
\end{abstract}



\maketitle

\section{Introduction}
Let $\Sigma^n$ be an $n$-dimensional submanifold isometrically immersed in an $m$-dimensional Riemannian manifold $(M^m,g_M)$. For simplicity, in this paper we always assume that $\Sigma$ is closed except special declaration. It is  well known that $\Sigma$ is minimal if and only if it is a critical point of  the volume functional. Precisely, if we consider a variation  $\Sigma_t \, (-\epsilon<t<\epsilon)$ of $\Sigma$ such that  $\Sigma_0=\Sigma$, and let $\vol(t)$ denote the volume of $\Sigma_t$, then 
\begin{equation}\label{eq-1var}
\vol'(0)=-\int_{\Sigma} \l n\mathbf{H}, V\r,
\end{equation}
here $\mathbf{H}$ is the mean curvature vector of $\Sigma$ in $M$, and $V$ is the normal variational  vector field on $\Sigma$.

A natural and important question is that whether the minimal submanifold yields a local minimum of the volume functional or not, namely, the minimal submanifold is stable or not.
So one needs to consider the second variation formula at the critical points (cf. \cite{Sim68, Che68}):
\begin{equation}\label{eq-2var}
\vol''(0)=\int_{\Sigma} \l \mathfrak{J}V, V\r,
\end{equation}
where the Jocobi operator $\mathfrak{J}$ 
is a strongly elliptic operator acting on the sections of the normal bundle $\Gamma(T^\bot\Sigma)$, and it has distinct, real eigenvalues
and the dimension of each eigenspace is finite.

The expression of $\mathfrak{J}$ will be presented in \eqref{jacobi}.
Here we just recall that a minimal submanifold $\Sigma$ is called \emph{stable} if and only if 
\begin{equation*}
I(V,V):=\int_\Sigma \l \mathfrak{J}V,V\r\ge 0, \,\forall V\in\Gamma(T^\bot \Sigma).
\end{equation*}

The simplest case is that the ambient space is a standard sphere.
J. Simons proved in his pioneering work \cite{Sim68} that, there are no closed minimal submanifold in the
standard sphere $S^{m}$.
Later, the existence and classification of closed stable minimal submanifolds in compact rank one symmetric spaces have been solved completely.
We summarize these results as follows.

\begin{thm}\label{thm-rank-one}
Let $\Sigma^n$ be a closed minimal submanifold immersed in $M^m$. Then $\Sigma$ is stable if and only if $\Sigma$ is one of the cases listed in the following table.
\begin{center}
\begin{tabular}{c|c|c|c|c|c}
	\rule[-1ex]{0pt}{3.5ex} &  \emph{Simons}\cite{Sim68} & \emph{Lawson and Simons} \cite{LS73} & \multicolumn{3}{c}{\emph{Ohnita} \cite{Ohn86a}} \\ 
	\hline 
	\rule[-1ex]{0pt}{3.5ex}  $M$ & $S^m$	& $P^n(\mathbb{C})$ & $P^n(\mathbb{R}) $& $P^n(\mathbb{H})$ & $P^2(\mathbf{Cay})$  \\ 
	\hline 
	\rule[-1ex]{0pt}{3.5ex} $\Sigma$ & not exist & complex submanifold	&  $P^l(\mathbb{R})$ &   $P^l(\mathbb{H})$& $P^1(\mathbf{Cay})$ \\ 
\end{tabular} 
\end{center}
\end{thm}

Noting that the $P^n(\mathbb{C})$ is $1/4$-pinched (i.e. the sectional curvature is in  $[1/4,1]$), based on Theorem \ref{thm-rank-one}, Lawson and Simons proposed the following conjecture:
\begin{con}[\cite{LS73}]
	There are no closed  stable minimal submanifolds in any compact, simply connected, strictly $1/4$-pinched (i.e. the sectional curvature is in  $(1/4,1]$) Riemannian manifold. 
\end{con}

It is also worth pointing out that, the condition in Lawson-Simons conjecture is closely related to the famous differentiable sphere theorem, i.e., a compact, simply connected, strictly $1/4$-pinched Riemannian manifold is diffeomorphic to the round unit sphere (\cite{BS09}).

Lawson-Simons conjecture is still open, but known results support this conjecture.
Howard and Wei \cite{HW15} verified the above conjecture for several classes of positively curved manifolds. On the other hand, when $M$ is a submanifolds of the Euclidean space $\R^N$, some nonexistence results of stable minimal submanifolds of $M$ were obtained by different people under different assumptions (e.g., $\delta$-pinched hypersurface in $\R^N$,  algebraic conditions on the second fundamental form of $M$ in $\R^N$), and in some results the pinching constant $\delta$ can be less than $1/4$ (cf. \cite{Sim68, Ohn86a, Li97, SH01, HW03, CW13, HW15}).
For the existence and nonexistence of stable submanifolds with fixed dimension $n$ (especially for $n=1,2$), the reader can refer \cite{Ami76, Zil77, Oka89, TU22}.

Very recently, Franz and Trinca studied the nonexistence of stable minimal submanifolds in a conformal sphere. They proved
\begin{thm}[\cite{FZ23}]\label{thm-FT}
	Let $(M^m,g)$ be a Riemannian manifold conformal to the round unit sphere $(S^m,h_1)$ and that is $\delta$-pinched for some $\delta>0$.
	Then there is no stable minimal $n$-dimensional submanifold of $M$ for all $2 \le n\le m-\delta^{-1}$.
\end{thm}

In this paper, we will generalize the above theorem to general ambient spaces rather than a sphere.
First, we introduce the following conformal invariant of a Riemannian manifold.
\begin{defn}
	For a closed Riemannian manifold $(M^m,g)$,
	we define a non-negative constant
	\begin{equation*}
		C(M,[g]):=\inf_{\psi\in \Gamma}\max_{M}|\sff_{\psi}|^2,
	\end{equation*}
	where
	\begin{equation*}
		\Gamma=\{\psi: M\to (S^{d},h_1)\mid \mbox{ $\psi$ is non-degenerate conformal map for some $d$}\},
	\end{equation*}
	and $\sff_{\psi}$ is the second fundamental form of $(M, \psi^\ast h_1)$ in $S^d$.
\end{defn}

\begin{rem}
	The conformal map $\psi$ exists for $d$ big enough because of the Nash embedding theorem (via the stereographic projection).

	If $(M^m,g)=(S^m,h_1)$, we have $C(M,[g])=0$.
\end{rem}

Throughout this article we set $\xi(m)=1$ when $3\le m\le 5$ and $\xi(m)=2$ when $m\ge 6$, now the main theorem is stated as follows.
\begin{thm}\label{thm:main-1}
	Let $(M^m,g)$ be an $m$-dimensional closed Riemannian manifold with non-negative sectional curvatures, $m\ge 3$.
	Then for each $\xi(m)\le n\le m-2$,
	there exists a constant $c(m,n)$ depending only on $m,n$ such that
	if
	\begin{equation}
		C(M,[g])< c(m,n),
	\end{equation}
	then
	there are no closed $n$-dimensional stable minimal submanifolds in $M$.
\end{thm}

We immediately obtain the following corollaries.
\begin{cor}\label{cor:1}
	Let $(M^m,g)$ be an $m (\ge 3)$-dimensional closed Riemannian manifold with non-negative sectional curvatures.
	Then there exists a constant $c'(m)$ depending only on $m$ such that
	if
	\begin{equation}
		C(M,[g])< c'(m),
	\end{equation}
	then
	there are no closed $n$-dimensional stable minimal submanifolds in $M$ for all $\xi(m)\le n\le m-2$.
\end{cor}

\begin{cor}\label{cor:2}
	If $(M^m,g)$ is an $m (\ge 3)$-dimensional closed Riemannian manifold with non-negative sectional curvatures and is conformal to $(S^m,h_1)$,
	then there are no closed $n$-dimensional stable minimal submanifolds in $M$ for all $\xi(m)\le n\le m-2$.
\end{cor}

\begin{cor}\label{cor:3}
	Let $(M^m,g)$ be an $m (\ge 3)$-dimensional closed submanifold in $S^d$ with non-negative sectional curvatures.
	If the second fundamental form $\sff$ of $M$ satisfies $|\sff|^2<c'(m)$,
	then there are no closed $n$-dimensional stable minimal submanifolds in $M$ for all $\xi(m)\le n\le m-2$.
\end{cor}

\begin{rem}
	(1) The Corollary \ref{cor:1} is directly from Theorem \ref{thm:main-1} by setting
	\begin{equation*}
		c'(m)=\min_{\xi(m)\le n\le m-2}\{c(m,n)\}.
	\end{equation*}

	(2) The Corollary \ref{cor:2} confirms Lawson–Simons conjecture in conformal spheres $S^m$ for all $n$-dimensional submanifolds with $2 \le n\le m-2$.
	
	(3) In fact, by Remark \ref{rem:rough-bounds}, we can set rough upper bounds
	\begin{equation}
		c(m,n)=\begin{cases}
			2-\frac{4}{m+1}, & \mbox{ if $n=1$};\\
			\frac{n(m-n)}{m}, & \mbox{ if $2\le n\le m-2$},
		\end{cases}
	\end{equation}
	and 
	\begin{equation}
		c'(m)=2-\frac{4}{m}.
	\end{equation}

	(4) If we already know that $(M^m,g)$ admits certain stable minimal $n$-dimensional submanifold, then $c(m,n)$ gives a lower bound of $C(M,[g])$ by Theorem \ref{thm:main-1}.

	(5) Although our results gives no information about $n=1$ and $p=1$, from Synge's theorem (cf. \cite[Section 6.3]{Pet16}) we know that any closed geodesic in an even dimensional oriented Riemannian manifold with positive sectional curvature must be unstable; it is not hard to show that any closed two-sided hypersurface in a Riemannian manifold with positive Ricci curvature is always unstable.
\end{rem}

\bigskip
This paper is organized as follows. In Sect.~\ref{Sect-2}, we introduce notations and give relations between corresponding geometric quantities of conformal metrics.
In Sect.~\ref{Sect-2}, we proof main theorem.
The key point is to choose suitable test normal vector fields.
In most of early work, the test normal vector fields are constructed by projecting an orthonormal basis of $\mathbb{R}^N$ to the $\Sigma$ when treated $M$ as a submanifold in $\mathbb{R}^N$, and these test normal vector fields have good property, especially for $M=S^m$.
We also consider these project vector fields.
But since $M$ is immersed in $S^d$ conformally rather than isometrically,
we need converting the computations w.r.t conformal metric to the original metric by using conformal relations.
More specifically, we modify the project vector fields by multiplying a factor involving a parameter to be determined.
It is worth to point out that, there are at least three differences from \cite{FZ23}.
First, we consider general ambient spaces rather than a sphere and use some algebraic inequalities to estimate the second fundamental form.
Second, we cancel terms involving derivatives along normal directions by Gauss equations, which makes us relax the $\delta$-pinched condition to non-negative sectional curvature.
Third, by choosing suitable parameter we can detail with the case $n=1$ when $3\le m\le 5$.

\section{Preliminaries}\label{Sect-2}

\subsection{Notations}
We agree on the following range of indices:
\begin{gather*}
	1\le i,j,k,\dots\le n; n+1\le \alpha, \beta, \gamma,\dots\le m;\\
	1\le A,B,C,\dots\le m; m+1\le \mu, \nu,\dots \le d;\\
	1\le s,t,\dots\le d+1.
\end{gather*}

For $(M^m,g_M)$, we use the following notations.
\begin{itemize}
	\item $\{e_A\}_{A=1}^{m}$: a local orthonormal frame on $(M, g_M)$;
	\item  $\{\omega_A\}_{A=1}^{m}$: the dual coframe of $\{e_A\}$;
	\item $\cm$: Levi-Civita connection on $(M,g_M)$;
	\item $\{\omega_{AB}\}$: the connection $1$-forms of $(M,g_M)$;
	\item $F_A$: the components of 1st covariant derivative of $F$ for any $F\in C^{\infty}(M)$;
	\item $F_{AB}$: the components of 2nd covariant derivative of $F$ for any $F\in C^{\infty}(M)$;
	\item $R^M_{ABCD}$: the components of $(0,4)$-type Riemannian curvature tensor on $(M,g_M)$, i.e.
	$R^M_{ABCD}=\langle (\cm_{e_A}\cm_{e_B}-\cm_{e_B}\cm_{e_A}-\cm_{[e_A,e_B]}) e_D, e_C \rangle$.
\end{itemize}

When considering an isometrical immersion $\phi: \Sigma^n\to (M^m,g_M)$, we have an induced metric
$g_{\Sigma}=\phi^\ast g_M$.
We choose $\{e_A\}=\{e_i,e_\alpha\}$, and denote $\theta_{A}=\phi^\ast \omega_{A}, \theta_{AB}=\phi^\ast \omega_{AB}$.
We know that $\{\theta_i\}$ is the dual of $\{e_i\}$ and $\theta_\alpha=0$. 
We set $p=m-n=\cod \Sigma$ and use the following notations.

\begin{itemize}
	\item $\cs$ and $\nc$: Levi-Civita connection and normal connection on $(\Sigma,g_\Sigma)$, resp.;
	\item $\theta_{ij}$ and $\theta_{\alpha\beta}$: Levi-Civita and normal connection $1$-forms on $(\Sigma,g_\Sigma)$, resp.;
	\item $R^{\Sigma}_{ijkl}$: the components of $(0,4)$-type Riemannian curvature tensor on $\Sigma$;
	\item $h_{ij}^{\alpha}$: the components of second fundamental form of $\Sigma$ in $(M,g_M)$, which satisfies $\theta_{i\alpha}=\sum_j h_{ij}^{\alpha}\theta_j$;
	\item $H^\alpha=
	\frac{1}{n}\sum_{i}h_{ii}^{\alpha}$: the components of mean curvature w.r.t $e_\alpha$ of $\Sigma$ in $(M,g_M)$;
	\item $|h|^2=\sum_{i,j,\alpha}(h_{ij}^\alpha)^2$: squared norm of the second fundamental form.
	\item $V^\alpha_{,j}$: the components of the 1st covariant derivative of $V=\sum_\alpha V^\alpha e_\alpha\in \Gamma(T^{\bot}\Sigma)$;
\end{itemize}

For any $F\in C^\infty(M)$, we consider the restriction of $F$ on $\Sigma$ (still denote by $F$).
Let $F_{i}^\Sigma, F_{ij}^\Sigma$ denote the components of the 1st and 2nd covariant derivatives of $F$ as a smooth function on $\Sigma$.
Then by pulling back
\begin{align}
	\sum_A F_{A}\omega_A&= dF,\label{eq:1st-d}\\
	\sum_{B}F_{AB}\omega_B&=dF_{A}+\sum_{B}\omega_{BA}F_B,\label{eq:2st-d}
\end{align}
by $\phi$ respectively, we obtain
\begin{align}
	F_i&=F_i^\Sigma,\label{eq:1st-d-s}\\
	F_{ij}&=F_{ij}^\Sigma-h_{ij}^\alpha F_{\alpha}\label{eq:2st-d-s}.
\end{align}
The squared norm of the gradient and the Laplacian of $F$ on $\Sigma$ are respectively given by
\begin{equation}\label{eq:2.5}
	|\cs F|^2=\sum_i F_i^2, \quad  \Delta^\Sigma F=\sum_{i} F^\Sigma_{ii}.
\end{equation}
The squared norm of the gradient and the Laplacian of $F$ on $M$ is given by
\begin{equation}\label{eq:2.6}
	|\cm F|^2=\sum_A F_A^2,
\end{equation}
and we will denote
\begin{equation}\label{eq:2.7}
	|\nc F|^2:=\sum_\alpha F_\alpha^2=|\cm F|^2-|\cs F|^2,
\end{equation}

\subsection{Conformal relations}
Let $\tgm$ be a metric on $M$ conforming to $g_M$. We assume $\tgm=e^{2u}g_M$ for some  $u\in C^{\infty}(M)$.
Let $\tgs=\phi^\ast\tgm$.

For convenience, we sometimes simply use $\tilde{M}$ and $\tilde{\Sigma}$ to represent $(M,\tgm)$ and $(\Sigma, \tgs)$ respectively (i.e., equipped with the conformal metrics), and $M$ and $\Sigma$ to represent the manifold equipped with the original metrics.

We add  ``$\tilde{\ }$'' to denote the corresponding quantities on $\tilde{M}$ and $\tilde{\Sigma}$.

We choose local frames such that $\{\et_A=e^{-u}e_A\}$ and $\{\tom_A=e^{u}\omega_A\}$.
Next we list some relations between corresponding geometric quantities of conformal metrics.
They are well known and the details of proofs can be found in related literatures (cf. \cite{Che73, CW19}).

\textbf{(1) Covariant derivatives of functions.}

From
\begin{equation*}
	\sum_{A=1}^{m}\tilde{F}_A\tom_A=dF=\sum_{A=1}^{m}F_A\omega_A
\end{equation*}
we have
\begin{equation}\label{eq27}
	\tilde{F}_A=e^{-u}F_A,~~\mbox{ for } A=1,\cdots, m.
\end{equation}

The structure equation implies
\begin{equation}
	\begin{aligned}\label{eq26}
	\tom_{AB}=\omega_{AB}+u_A\omega_B-u_B\omega_A.
	\end{aligned}
\end{equation}

Since
\begin{equation*}
	\sum_{B=1}^m F_{AB}\omega_B=dF_A+\sum_{B=1}^m\omega_{BA}F_B,
\end{equation*}
we have
\begin{equation*}
	e^{2u}\tilde{F}_{AB}=F_{AB}+\sum_{C}u_CF_C \delta_{AB}-F_Au_B-F_Bu_A.
\end{equation*}

\textbf{(2) Riemannian curvatures.}

\begin{equation}\label{eq209}
	\begin{aligned}
	e^{2u}R^{\tilde{M}}_{ABCD}={}&R^M_{ABCD}-(u_{AC}\delta_{BD}+u_{BD}\delta_{AC}-u_{AD}\delta_{BC}-u_{BC}\delta_{AD})\\
	&+(u_Au_C\delta_{BD}+u_Bu_D\delta_{AC}-u_Bu_C\delta_{AD}-u_Au_D\delta_{BC})\\
	&-(\sum_Cu_{C}^2)(\delta_{AC}\delta_{BD}-\delta_{AD}\delta_{BC}).
	\end{aligned}
\end{equation}

For $A\neq B$, it follows from \eqref{eq209} that
\begin{equation*}
	\begin{aligned}
		e^{-2u}R^M_{ABAB}={}&R^{\tilde{M}}_{ABAB}+(\tu_{AA}+\tu_{BB})+(\tu_A^2+\tu_B^2)-|\tcm u|_{\tg}^2\\
		={}&1+\langle \sff(\et_A,\et_A), \sff(\et_B,\et_B)\rangle-\langle \sff(\et_A,\et_B), \sff(\et_A,\et_B)\rangle\\
		&+(\tu_{AA}+\tu_{BB})+(\tu_A^2+\tu_B^2)-|\tcm u|_{\tg}^2.
	\end{aligned}
\end{equation*}
In particular,
\begin{align}
	\sum_{i,\alpha}e^{-2u}R^M_{i\alpha i\alpha}
		&=\sum_{i,\alpha}R^{\tilde{M}}_{i\alpha i\alpha}+p\sum_{i}\tu_{ii}+n\sum_{\alpha}\tu_{\alpha\alpha}+p\sum_{i}\tu_i^2+n\sum_{\alpha}\tu_\alpha^2-np|\tcm u|_{\tg}^2,\label{eq:sec-t-n}\\
	\sum_{\alpha\neq \beta}e^{-2u}R^M_{\alpha\beta\alpha\beta}
		&=\sum_{\alpha\neq \beta}R^{\tilde{M}}_{\alpha\beta\alpha\beta}
		+2(p-1)\sum_{\alpha}\tu_{\alpha\alpha}+2(p-1)\sum_{\alpha}\tu_\alpha^2-p(p-1)|\tcm u|_{\tg}^2,\label{eq:sec-n-n}\\
	\sum_{i\neq j}e^{-2u}R^M_{ijij}
		&=\sum_{i\neq j}R^{\tilde{M}}_{ijij}
		+2(n-1)\sum_{i}\tu_{ii}+2(n-1)\sum_{i}\tu_i^2-n(n-1)|\tcm u|_{\tg}^2.\label{eq:sec-t-t}
\end{align}

\textbf{(3) Second fundamental forms}
\begin{equation}\label{eq-208}
	\th_{ij}^{\alpha}=e^{-u}(h_{ij}^{\alpha}-u_{\alpha}\delta_{ij}), \quad\tH^{\alpha}=e^{-u}(H^{\alpha}-u_{\alpha}).
\end{equation}

\textbf{(4) Covariant derivatives of vector fields.}

\begin{lem}
	For $ V=\Gamma(T^{\bot}\Sigma)$, we have
	\begin{align}
		\tV^\alpha_{,j}&=V^\alpha_{,j}+V^\alpha u_j.\label{eq:V-conf-1st}
	\end{align}
\end{lem}
\begin{proof}
	Pulling \eqref{eq26} back by $\phi$, we get $\theta_{\alpha\beta}=\tt_{\alpha\beta}$.
	Then \eqref{eq:V-conf-1st} is from
	\begin{gather*}
		\tV^\alpha=e^u V^\alpha,\quad \tt_j=e^u \theta_j,\\
		\sum_{j}V^\alpha_{,j} \theta_j=dV^\alpha+\sum_{\beta}\theta_{\beta\alpha}V^\beta,\quad \sum_{j}\tV^\alpha_{,j} \tt_j=d\tV^\alpha+\sum_{\beta}\tt_{\beta\alpha}\tV^\beta.
	\end{gather*}
\end{proof}

\subsection{Some calculations}\label{Sect-2.3}
It is well known that the Jacobi operator is
\begin{gather}
	\mathfrak{J}=-\big(\Delta^\bot+\mathfrak{B}+\mathfrak{R}\big),\label{jacobi}\\
\Delta^\bot V=\sum_{i=1}^{n}\big(\nabla^\bot_{e_i}\nabla^\bot_{e_i}-\nabla^\bot_{\nabla^{\Sigma}_{e_i}e_i}\big) V ,\quad \mathfrak{B}(V)=\sum_{i=1}^{n}h(e_i, A_{V}e_i),\quad \mathfrak{R}(V)=\sum_{i=1}^{n}(R^M(V,e_i)e_i)^\bot,\nonumber
\end{gather}
By using notations above, we have
\begin{equation*}
	I(V,V)=\int_{\Sigma}\Big(|\nc V|_{g_\Sigma}^2-\sum_{i,\alpha,\beta}R^M_{i\alpha i\beta}V^\alpha V^\beta-\sum_{i,j,\alpha,\beta}h_{ij}^{\alpha}h_{ij}^{\beta}V^\alpha V^\beta\Big)\,dv_g.
\end{equation*}

Next, we replace the integrand by quantities w.r.t. the conformal metric $\tg=e^{2u}g$.

\begin{equation}
	\begin{aligned}
		|\nc V|_{g}^2&=\sum_{j,\alpha}(V^\alpha_{,j})^2=\sum_{j,\alpha}(\tV^\alpha_{,j}-\tV^\alpha \tu_{j})^2\\
		&=|\tnc V|_{\tg}^2-\sum_{j,\alpha}2\tV^\alpha_{,j}\tV^\alpha \tu_{j}+|V|_{\tg}^2|\tcs u|_{\tg}^2.
	\end{aligned}
\end{equation}

\begin{equation}
	\begin{aligned}
		\sum_{i,\alpha,\beta}R^M_{i\alpha i\beta}V^\alpha V^\beta={}&\sum_{i,\alpha,\beta}\Big(R^{\tilde{M}}_{i\alpha i\beta}+(\tu_{ii}\delta_{\alpha\beta}+\tu_{\alpha\beta}\delta_{ii})+(\tu_i\tu_i\delta_{\alpha\beta}+\tu_\alpha\tu_\beta\delta_{ii})\\
		&{}\qquad-(\sum_C\tu_{C}^2)\delta_{\alpha\beta}\delta_{ii}\Big)\tV^\alpha \tV^\beta\\
		={}&\sum_{i,\alpha,\beta}R^{\tilde{M}}_{i\alpha i\beta}\tV^\alpha \tV^\beta
		+\sum_{i} \tu_{ii}|V|_{\tg}^2+|\tcs u|_{\tg}^2|V|_{\tg}^2-n|\tcm u|_{\tg}^2|V|_{\tg}^2\\
		&\qquad+n\sum_{\alpha,\beta}\tu_{\alpha\beta}\tV^\alpha \tV^\beta+n\sum_{\alpha,\beta}\tu_\alpha\tu_\beta\tV^\alpha \tV^\beta.
	\end{aligned}
\end{equation}

If $H^{\alpha}=0$, then $\tilde{H}^\alpha=-e^{-u}u_{\alpha}=-\tu_\alpha$.
Hence,
\begin{equation}
	\begin{aligned}
		\sum_{i,j,\alpha,\beta}h_{ij}^{\alpha}h_{ij}^{\beta}V^\alpha V^\beta={}&\sum_{i,j,\alpha,\beta}(\th_{ij}^{\alpha}+\tu_{\alpha}\delta_{ij})(\th_{ij}^{\beta}+\tu_{\beta}\delta_{ij})\tV^\alpha \tV^\beta\\
		={}&\sum_{i,j,\alpha,\beta}\th_{ij}^{\alpha}\th_{ij}^{\beta}\tV^\alpha \tV^\beta
		+n\sum_{\alpha,\beta}\tu_\alpha\tu_\beta\tV^\alpha \tV^\beta\\
		&+n\sum_{\alpha,\beta}\tilde{H}^\alpha\tu_\beta\tV^\alpha \tV^\beta+n\sum_{\alpha,\beta}\tilde{H}^\beta\tu_\alpha\tV^\alpha \tV^\beta\\
		={}&\sum_{i,j,\alpha,\beta}\th_{ij}^{\alpha}\th_{ij}^{\beta}\tV^\alpha \tV^\beta
		-n\sum_{\alpha,\beta}\tu_\alpha\tu_\beta\tV^\alpha \tV^\beta.
	\end{aligned}
\end{equation}
Hence, we obtain the following
\begin{lem}\label{lem:int-conf}
	\begin{equation}\label{eq:int-conf}
		\begin{aligned}
			\langle \Joc V, V\rangle
			={}&|\tnc V|_{\tg}^2-\sum_{i,\alpha,\beta}R^{\tilde{M}}_{i\alpha i\beta}\tV^\alpha \tV^\beta
			-\sum_{i,j,\alpha,\beta}\th_{ij}^{\alpha}\th_{ij}^{\beta}\tV^\alpha \tV^\beta\\
			&-\sum_{j,\alpha}2\tV^\alpha_{,j}\tV^\alpha \tu_{j}-\sum_{i} \tu_{ii}|V|_{\tg}^2+n|\tcm u|_{\tg}^2|V|_{\tg}^2-n\sum_{\alpha,\beta}\tu_{\alpha\beta}\tV^\alpha \tV^\beta.
		\end{aligned}
	\end{equation}
\end{lem}

\section{Proof of main theorems}\label{Sect-3}
In this section, we prove Theorems \ref{thm:main-1}.
The chain of maps
\begin{equation*}
	(\Sigma,g_\Sigma) \xrightarrow{\phi} (M,g_M) \xrightarrow{\psi} (S^d,h_1)\subset \mathbb{R}^{d+1}
\end{equation*}
gives conformal metrics $\tgm=\psi^\ast g_M=e^{2u}g_M$ and $\tgs=\phi^\ast \tgm$ on $M$ and $\Sigma$, respectively.
Here $\phi$ is a minimal isometric immersion, and $\psi$ is a conformal map.

For any fixed vector $E\in \mathbb{R}^{d+1}$,
we construct a test normal section $V_E=\sum_\alpha \tV^\alpha \et_\alpha$ such that $\tV^\alpha=e^{au}\langle E, \et_\alpha\rangle$.
Intuitively,  $V_E$ is the projection of $E$ to $\Gamma(T^{\bot}\tilde{\Sigma})$ in $\Gamma(T^{\bot}\tilde{M})$ multiplying a factor $e^{au}$ w.r.t. $\tg$.
\begin{lem}
	For $V_E$ defined above, we have
	\begin{equation*}
		\tV^{\alpha}_{,j}=e^{au}\Big(-\sum_i \th_{ij}^\alpha\langle E, \et_i\rangle+\langle E, \sff(\et_j,\et_\alpha)\rangle +a\tu_j\langle E, \et_\alpha\rangle\Big)
	\end{equation*}
\end{lem}
\begin{proof}
	A direct computation shows that
	\begin{equation*}
		\begin{aligned}
			\sum_{\alpha}\tV^{\alpha}_{,j}\et_\alpha={}&\tnc_{\et_j}V_E=\Big(\sum_{\alpha}e^{au}\langle E, \et_\alpha\rangle \et_\alpha\Big)\\
			={}&\sum_{\alpha}e^{au}\Big(\langle E, \nabla^{\mathbb{R}^{d+1}}_{\et_j}\et_\alpha\rangle \et_\alpha+\langle E, \et_\alpha\rangle \tnc_{\et_j}\et_\alpha+a\tu_j\langle E, \et_\alpha\rangle\Big)\\
			={}&e^{au}\sum_{\alpha}\big\langle E_t, -A_{\et_\alpha}(\et_j)+\sum_\beta\tt_{\alpha\beta}(\et_j)\et_\beta+\sff(\et_j,\et_\alpha)\big\rangle \et_\alpha\\
			&+\sum_{\alpha}e^{au}\langle E, \et_\alpha\rangle \sum_\beta\tt_{\alpha\beta}(\et_j)\et_\beta+\sum_{\alpha}e^{au}a\tu_j\langle E, \et_\alpha\rangle\\
			={}&\sum_{\alpha}e^{au}\Big(-\sum_i\th_{ij}^\alpha\langle E, \et_i\rangle+\langle E, \sff(\et_j,\et_\alpha)+a\tu_j\langle E, \et_\alpha\rangle\rangle\Big)\et_\alpha.
		\end{aligned}
	\end{equation*}
	So the lemma has been proved.
\end{proof}

Let $\{E_t\}$ be an orthonormal basis of $\mathbb{R}^{d+1}$, and define $V_t=V_{E_t}$.
We have
\begin{prop}\label{prop3.3}
	\begin{equation}\label{eq:prop3.3}
		e^{-2au}\sum_{t}\langle \Joc V_t, V_t\rangle=\sum_{i,\alpha}|\sff(\et_i,\et_\alpha)|^2+p(a-1)^2\sum_{i}\tu_i^2+n\sum_{\alpha}\tu_\alpha^2-\sum_{i,\alpha}e^{-2u}R^M_{i\alpha i\alpha}.
	\end{equation}
\end{prop}
\begin{proof}
	By noting the fact $\sum_{t}\langle E_t, \bar U\rangle \langle E_t, \bar V\rangle=\langle \bar U, \bar V\rangle$ for any vector fields $\bar U, \bar V$ in $\mathbb{R}^{d+1}$,
	we have
	\begin{gather}
		\sum_{t}|\tnc V_t|_{\tg}^2=e^{2au}\Big(\sum_{i,j,\alpha}(\th_{ij}^\alpha)^2+\sum_{j,\alpha}|\sff(\et_j,\et_\alpha)|^2+ pa^2\sum_j \tu_j^2\Big),\\
		\sum_{t,i,\alpha,\beta}R^{\tilde{M}}_{i\alpha i\beta}\tV_t^\alpha \tV_t^\beta=e^{2au}\sum_{i,\alpha}
		R^{\tilde{M}}_{i\alpha i\alpha},
		\quad
		\sum_{t, i,j,\alpha,\beta}\th_{ij}^{\alpha}\th_{ij}^{\beta}\tV_t^\alpha \tV_t^\beta=e^{2au}\sum_{i,j,\alpha}(\th_{ij}^\alpha)^2,\\
		\sum_{t,j,\alpha}\tV^\alpha_{t,j}\tV_t^\alpha \tu_{j}=e^{2au}pa\sum_{j}\tu_j^2,
		\quad
		\sum_{t}|V_t|_{\tg}^2=e^{2au}p,
		\quad \sum_{t,\alpha,\beta}\tu_{\alpha\beta}\tV_t^\alpha\tV_t^\beta=e^{2au}\sum_{\alpha}\tu_{\alpha\alpha},
	\end{gather}
	Putting above equalities into \eqref{eq:int-conf},
	we obtain
	\begin{equation}
		\begin{aligned}
			e^{-2au}\sum_{t}\langle \Joc V_t, V_t\rangle={}&\sum_{i,\alpha}|\sff(\et_i,\et_\alpha)|^2+ pa^2\sum_j \tu_j^2-\sum_{i,\alpha}R^{\tilde{M}}_{i\alpha i\alpha}\\
			{}&-2pa\sum_{j}\tu_j^2-p\sum_{i}\tu_{ii}+np|\tcm u|^2_{\tg}-n\sum_{\alpha}\tu_{\alpha\alpha}\\
			={}&\sum_{i,\alpha}|\sff(\et_i,\et_\alpha)|^2-e^{-2u}\sum_{i,\alpha}R^{M}_{i\alpha i\alpha}+p(a-1)^2\sum_{j}\tu_j^2+n\sum_{\alpha}\tu_\alpha^2,
		\end{aligned}
	\end{equation}
	where we used  \eqref{eq:sec-t-n} in the last equality.
\end{proof}

\begin{prop}\label{prop3.4}
	For $p\ge 2$, we have
	\begin{align*}
		\sum_{t}I(V_t,V_t)
		={}&\int_{\Sigma}\big(\mathcal{F}(\sff)-n\big)e^{2au}\,dv_g\\
		&+\int_{\Sigma}\Big(\frac{2-p}{p}\sum_{i,\alpha}R^M_{i\alpha i\alpha}-\frac{n}{p(p-1)}\sum_{\alpha\neq \beta}R^M_{\alpha\beta\alpha\beta}\Big)e^{2(a-1)u}\,dv_g\\
		&+\int_{\Sigma}\big(4a-2-n+p(a-1)^2\big)|\tcs u|_{\tg}^2\,dv_g.
	\end{align*}
	where
	\begin{equation}
		\begin{aligned}
			\mathcal{F}(\sff)=&\Big(1+\frac{2}{p}\Big)\sum_{i,\alpha}|\sff(\et_i,\et_\alpha)|^2-\frac{2}{p}\sum_{i,\alpha}\langle \sff(\et_i, \et_i), \sff(\et_\alpha, \et_\alpha)\rangle
			\\
			&\quad +\frac{n}{p(p-1)}\sum_{\alpha\neq \beta}\big(\langle \sff(\et_\alpha, \et_\alpha), \sff(\et_\beta, \et_\beta)\rangle-|\sff(\et_\alpha, \et_\beta)|^2\big).
		\end{aligned}
	\end{equation}
\end{prop}
\begin{proof}
Recall
\begin{equation}\label{eq:2nd-der-restriction}
	\sum_{i}\tu_{ii}=\sum_{i}\tu_{ii}^{\Sigma}-\sum_{\alpha}n\tilde{H}^\alpha\tu_{\alpha}=\sum_{i}\tu_{ii}^{\Sigma}+n\sum_{\alpha}\tu_{\alpha}^2,
\end{equation}
then we can rewrite 
\eqref{eq:sec-t-n}, \eqref{eq:sec-n-n} and \eqref{eq:sec-t-t} as (cf. \eqref{eq:2st-d-s}, \eqref{eq:2.5}, \eqref{eq:2.6} and \eqref{eq:2.7})
\begin{align}
	\sum_{i,\alpha}e^{-2u}R^M_{i\alpha i\alpha}
		&=\sum_{i,\alpha}R^{\tilde{M}}_{i\alpha i\alpha}+p\sum_{i}\tu_{ii}^{\Sigma}+n\sum_{\alpha}\tu_{\alpha\alpha}
		+p(1-n)|\tcs u|_{\tg}^2+n|\tnc u|_{\tg}^2,\label{eq:sec-t-n'}\\
	\sum_{\alpha\neq \beta}e^{-2u}R^M_{\alpha\beta\alpha\beta}
		&=\sum_{\alpha\neq \beta}R^{\tilde{M}}_{\alpha\beta\alpha\beta}
		+2(p-1)\sum_{\alpha}\tu_{\alpha\alpha}-p(p-1)|\tcs u|_{\tg}^2-(p-1)(p-2)|\tnc u|_{\tg}^2,\label{eq:sec-n-n'}\\
	\sum_{i\neq j}e^{-2u}R^M_{ijij}
		&=\sum_{i\neq j}R^{\tilde{M}}_{ijij}
		+2(n-1)\sum_{i}\tu_{ii}^{\Sigma}-(n-1)(n-2)|\tcs u|_{\tg}^2+n(n-1)|\tnc u|_{\tg}^2.\label{eq:sec-t-t'}
\end{align}

When $p\ge 2$, we cancel $\sum_{\alpha}\tu_{\alpha\alpha}$ from \eqref{eq:sec-t-n'} and \eqref{eq:sec-n-n'} and then solve out
\begin{align}
	|\tnc u|_{\tg}^2={}&\Big(\frac{1}{p(p-1)}\sum_{\alpha\neq \beta}R^{\tilde{M}}_{\alpha\beta\alpha\beta}-\frac{2}{np}\sum_{i,\alpha}R^{\tilde{M}}_{i\alpha i\alpha}\Big)\\
	&-\Big(\frac{1}{p(p-1)}\sum_{\alpha\neq \beta}e^{-2u}R^M_{\alpha\beta\alpha\beta}-\frac{2}{np}\sum_{i,\alpha}e^{-2u}R^M_{i\alpha i\alpha}\Big)\\
		&-\frac{2}{n}\sum_{i}\tu_{ii}^{\Sigma}
		+\frac{n-2}{n}|\tcs u|_{\tg}^2.
\end{align}
Putting this into \eqref{eq:prop3.3}, we have
\begin{equation}\label{eq:3.17}
	\begin{aligned}
		e^{-2au}\sum_{t}\langle \Joc V_t, V_t\rangle={}&\sum_{i,\alpha}|\sff(\et_i,\et_\alpha)|^2+\Big(\frac{n}{p(p-1)}\sum_{\alpha\neq \beta}R^{\tilde{M}}_{\alpha\beta\alpha\beta}-\frac{2}{p}\sum_{i,\alpha}R^{\tilde{M}}_{i\alpha i\alpha}\Big)\\
		&+\frac{2-p}{p}\sum_{i,\alpha}e^{-2u}R^M_{i\alpha i\alpha}-\frac{n}{p(p-1)}\sum_{\alpha\neq \beta}e^{-2u}R^M_{\alpha\beta\alpha\beta}\\
			&-2\sum_{i}\tu_{ii}^{\Sigma}
			+\big((n-2)+p(a-1)^2\big)|\tcs u|_{\tg}^2.
	\end{aligned}
\end{equation}

By Gauss equation and Stokes' formula, we have
\begin{align*}
	R^{\tilde{M}}_{\alpha\beta\alpha\beta}={}&1+\langle \sff(\et_\alpha, \et_\alpha), \sff(\et_\beta, \et_\beta)\rangle-|\sff(\et_\alpha, \et_\beta)|^2   \mbox{ for } \alpha\neq \beta,\\
	R^{\tilde{M}}_{i\alpha i\alpha}={}&1+\langle \sff(\et_i, \et_i), \sff(\et_\alpha, \et_\alpha)\rangle-|\sff(\et_i, \et_\alpha)|^2,\\
	\int_\Sigma \sum_{i}\tu_{ii}^{\Sigma}e^{2au}\,dv_g&=\int_\Sigma \sum_{i}\tu_{ii}^{\Sigma}e^{(2a-n)u}\,dv_{\tg}\\
	&=(n-2a)\int_\Sigma |\tcs u|^2 e^{(2a-n)u}\,dv_{\tg}\\
	&=(n-2a)\int_\Sigma |\tcs u|^2e^{2au} \,dv_{g},
\end{align*}
Therefore, we complete the proof by integrating \eqref{eq:3.17} over $(\Sigma, g)$.
\end{proof}

\begin{proof}[Proof of Theorem \ref{thm:main-1}]
	Since $C(M,[g])<c(m,n)$, by the definition we can find a conformal map $\psi: M\to S^d$ such that
	$\max |\sff_\psi|^2<c(m,n)$. We will omit the subscript $\psi$.

	Note Proposition \ref{prop3.4},
	it is sufficient to show that
	\begin{equation*}
		\sum_{t}I(V_t,V_t)<0.
	\end{equation*}
	Firstly, since $p\ge 2$, the sectional curvatures are nonnegative implies
	\begin{equation*}
		\int_{\Sigma}\Big(\frac{2-p}{p}\sum_{i,\alpha}R^M_{i\alpha i\alpha}-\frac{n}{p(p-1)}\sum_{\alpha\neq \beta}R^M_{\alpha\beta\alpha\beta}\Big)e^{2(a-1)u}\,dv_g\le 0
	\end{equation*}
	for all $n\ge 1$ and arbitrary parameter $a$.
	
	Secondly, when $n\ge 2$, we take $a=1$, then
	\begin{equation*}
		4a-2-n+p(a-1)^2=-n+2\le 0;
	\end{equation*}
	When $n=1$, we take $a=1/2$, then
	\begin{equation*}
		4a-2-n+p(a-1)^2=p/4-1\le 0
	\end{equation*}
	provided $p\le 4$ (i.e. $m=n+p\le 5$).
	Hence,
	\begin{equation*}
		\int_{\Sigma}\big(4a-2-n+p(a-1)^2\big)|\tcs u|_{\tg}^2\,dv_g\le 0
	\end{equation*}
	under the assumptions.

	At last, by using the following Lemma \ref{lem-II} and by setting $c(m,n)=\frac{n}{c_1(m,n)}$, we have
	\begin{equation*}
		\int_{\Sigma}\big(\mathcal{F}(\sff)-n\big)e^{2au}\,dv_g<0.
	\end{equation*}
	Hence, we complete the whole proof.
\end{proof}

\begin{lem}\label{lem-II}
	For each $1\le n\le m-2$, there exists a constant $c_1=c_1(m,n)$ depending only on $m,n$ such that
	\begin{equation}
		\mathcal{F}(\sff)\le c_1(m,n)|\sff|^2.
	\end{equation}
\end{lem}
\begin{proof}
	Let $\{\et_i, \et_\alpha, \et_{\mu}\}$ be a local orthonormal frame on $(S^d,h_1)$ such that $\{\et_{\mu}\}$ normal to $M^m$.
	Denote $\tb_{AB}^\mu=\langle \sff(\et_A,\et_B), \et_{\mu}\rangle$.
	By Cauchy-Schwarz inequality, it is not hard to prove that for any $k\times k$ symmetric real matrix $\Omega$, we have
	\begin{equation*}
		(\tr \Omega)^2\le k \tr (\Omega^T\Omega).
	\end{equation*}
	Hence,
	\begin{equation}
		\begin{aligned}
			&\sum_{\alpha\neq \beta}\big(\langle \sff(\et_\alpha, \et_\alpha), \sff(\et_\beta, \et_\beta)\rangle-|\sff(\et_\alpha, \et_\beta)|^2\big)\\
			={}&\sum_{\alpha,\beta}\big(\langle \sff(\et_\alpha, \et_\alpha), \sff(\et_\beta, \et_\beta)\rangle-|\sff(\et_\alpha, \et_\beta)|^2\big)\\
			={}&\sum_{\alpha,\beta,\mu}\big( b_{\alpha\alpha}^\mu b_{\beta\beta}^\mu-(b_{\alpha\beta}^\mu)^2\big)\le (p-1)\sum_{\alpha,\beta,\mu}(b_{\alpha\beta}^\mu)^2,\\
			&\Big|\sum_{i,\alpha}\langle \sff(\et_i, \et_i), \sff(\et_\alpha, \et_\alpha)\rangle\Big|=\Big|\sum_{\mu}\big(\sum_i b_{ii}^\mu \big)\big(\sum_{\alpha}b_{\alpha\alpha}^\mu\big)\Big|\\
			\le {}&\sum_{\mu}\frac{1}{2}\Big(\epsilon\big(\sum_i b_{ii}^\mu \big)^2+\frac{1}{\epsilon}\big(\sum_{\alpha}b_{\alpha\alpha}^\mu\big)^2\Big)\\
			\le{}&\sum_{\mu}\frac{1}{2}\Big(n\epsilon \sum_{i,j} (b_{ij}^\mu)^2+\frac{p}{\epsilon}\sum_{\alpha,\beta}(b_{\alpha\beta}^\mu)^2\Big).
		\end{aligned}
	\end{equation}
	We derive that
	\begin{equation}\label{eq:3.20}
		\begin{aligned}
			\mathcal{F}(\sff)\le {} &\Big(1+\frac{2}{p}\Big)\sum_{i,\alpha,\mu}(b_{i\alpha}^\mu)^2+\frac{n\epsilon}{p}\sum_{i,j,\mu} (b_{ij}^\mu)^2+\frac{1}{\epsilon}\sum_{\alpha,\beta,\mu}(b_{\alpha\beta}^\mu)^2+\frac{n}{p}\sum_{\alpha,\beta,\mu}(b_{\alpha\beta}^\mu)^2\\
			= {} &\frac{1}{p}\Big((\frac{p}{2}+1)2\sum_{i,\alpha,\mu}(b_{i\alpha}^\mu)^2+n\epsilon \sum_{i,j,\mu} (b_{ij}^\mu)^2+(n+\frac{p}{\epsilon})\sum_{\alpha,\beta,\mu}(b_{\alpha\beta}^\mu)^2\Big).
		\end{aligned}
	\end{equation}
	
	For any fixed $n$ (so $p=m-n$ is also fixed), 
	let $\epsilon_0$ be the positive root of $n\epsilon=n+p/\epsilon$, and set $c_2(m,n)=\max\{1+p/2,n\epsilon_0\}$.
	Then from \eqref{eq:3.20} we have
	\begin{equation*}
			\mathcal{F}(\sff)\le \frac{c_2(m,n)}{p}\Big(2\sum_{i,\alpha,\mu}(b_{i\alpha}^\mu)^2+\sum_{i,j,\mu} (b_{ij}^\mu)^2+\sum_{\alpha,\beta,\mu}(b_{\alpha\beta}^\mu)^2\Big)
			=c_1(m,n)|\sff|^2,
	\end{equation*}
	where $c_1(m,n)=\frac{c_2(m,n)}{m-n}$.
\end{proof}

\begin{rem}\label{rem:rough-bounds}
	We can choose $\epsilon$ in \eqref{eq:3.20} and get some rough bounds.
	For instance, when $n=1$, we take $\epsilon=2$. Then
	\begin{equation*}
		\mathcal{F}(\sff)\le \frac{m+1}{2(m-1)}|\sff|^2.
	\end{equation*}
	When $n\ge 2$, we simply take $\epsilon=1$. Then
	\begin{equation*}
			\mathcal{F}(\sff)\le \frac{m}{m-n}|\sff|^2.
	\end{equation*}
\end{rem}

In the end, we give an example to shows that even if $M^m$ is a hypersurface of $\mathbb{R}^{m+1}$, our theorems are not implied by previous results for that $M^m$ is $\delta$-pinched.
\begin{ex}\label{ex:ellipsoid}
	We consider a $4$-dimensional ellipsoid in $\mathbb{R}^5$ as follows
	\begin{equation}
		M^4=\Big\{(x_1,\dots,x_5)\in \mathbb{R}^5 \mid  x_1^2+\dots+x_4^2+\frac{x_5^5}{a^2}=1, a>0\Big\}.
	\end{equation}

	The sectional curvatures of $M$ satisfies (cf. \cite{SP86})
	\begin{equation}
		a^2\le K_M\le 1/a^4
	\end{equation}
	for $0<a<1$ and
	\begin{equation}
		\frac{1}{a^4}\le K_M\le a^2 
	\end{equation}
	for $a>1$.
	Hence, $M$ is $\delta$-pinched with $\delta=a^6$ for $0<a<1$ and $\delta=1/a^6$ for $a>1$.
	
	By \cite[Theorem 1.3]{CW13}, it is required that $\delta=1/\sqrt{m+1}=1/\sqrt{5}$, which implies $a$ must satisfy $ 0.88\approx(1/5)^{1/12}\le a\le 5^{1/12}\approx 1.14$.

	On the other hand, the standard metric $h_0$ of $\mathbb{R}^5$ is $h_0=\sum_{1\le A\le 5} (dx_A)^2$.
	By stereographic projection $\mathbb{R}^5\to S^5$, the conformal metric on $\mathbb{R}^5$
	is given by $e^{2u}h_0$
	with
	\begin{equation}
		e^u=\frac{2}{1+|x|^2}.
	\end{equation}
	Restricting this conformal metric to $M$, we obtain an isometric immersion $(M^4,\tg_M) \to S^5$.
	By Appendix A, the second fundamental form $\sff$ of $(M^4,\tg_M)$ in $S^5$ satisfies
	\begin{equation}
		\max |\sff|^2=\begin{cases}
			(1/a^2-1)^2, & \mbox{for $0<a\le 1$};\\
			(a^2-1)^2 a^2, & \mbox{for $a>1$}.
		\end{cases}
	\end{equation}
	For $m=4$, we can take
\begin{equation}
	c(4,1)=6/5, \quad c(4,2)=\sqrt{5}-1\approx 1.236;\quad c'(4)=6/5.
\end{equation}

	Hence, considering Corollary \ref{cor:1}, $\max |\sff|^2<6/5$ provided $0.691 \approx a_1<a<a_2 \approx 1.346$.
\end{ex}

\appendix
\section{The conformal second fundamental form of ellipsoids}
In this appendix, we give the details of Example \ref{ex:ellipsoid}.
We consider $n$-ellipsoid
\begin{equation}
	M^n=\Big\{(x^1,\dots,x^{n+1})\in \mathbb{R}^{n+1} \mid  (x^1)^2+\dots+(x^n)^2+\frac{(x^{n+1})^2}{a^2}=1, a>0\Big\}.
\end{equation}
First, we compute the second fundamental form of $M^n$ in $\mathbb{R}^{n+1}$.
By symmetry and continuity, we only need consider $M_{+}=\{x\in M\mid x^{n+1}>0\}$.

For a graph $x^{n+1}=a\sqrt{1-\sum_{i=1}^n (x^i)^2}=:f(x^1,\dots,x^n)$.
We choose the unit normal vector
\begin{equation}
	e_{n+1}=\frac{1}{w}\Big(\sum_{i=1}^n \frac{\partial}{\partial x^i}-\frac{\partial}{\partial x^{n+1}}\Big),
\end{equation}
then it is well known that (cf. \cite{SP86})
the induced metric $g=\sum_{1\le i, j\le n}g_{ij}dx^i dx^j$ and the second fundamental form $h=\sum_{1\le i, j\le n}h_{ij}dx^i dx^j$ are respectively given by
\begin{equation}
	g_{ij}=\delta_{ij}+f_i f_j,\quad
	h_{ij}=-\frac{1}{w}f_{ij},
\end{equation}
where
\begin{gather}
	f_i=\frac{\partial f}{\partial x^i}=-\frac{a^2}{f}x^i,\quad
	f_{ij}=\frac{\partial^2 f}{\partial x^i \partial  x^j}=-\frac{a^2\delta_{ij}+f_i f_j}{f},\\
	w=\sqrt{1+|\nabla f|^2}, \quad
	h_{ij}=\frac{1}{wf}\Big(a^2\delta_{ij}+f_i f_j\Big),\\
	|\nabla f|^2=\sum_{i=1}^n f_i^2=a^2\Big(\frac{a^2}{f^2}-1\Big).
\end{gather}
Then the mean curvature is given by
\begin{equation}
	nH=\sum_{1\le i,j \le n}(g^{ij}h_{ij})=\frac{1}{w f}\Big(n a^2 +\frac{1-a^2}{w^2}|\nabla f|^2\Big),
\end{equation}
and the squared norm of the second fundamental form is given by
\begin{equation}
	|h|^2_g=\sum_{1\le i,j,k,l\le 4}g^{ik}g^{jl}h_{ij}h_{kl}=\frac{1}{w^2 f^2}\Big(na^4 +\big(\frac{1-a^2}{w^2}\big)^2|\nabla f|^4+2a^2 \frac{1-a^2}{w^2}|\nabla f|^2\Big),
\end{equation}
where
\begin{equation}
	g^{ij}=\delta_{ij}-\frac{f_i f_j}{w^2}.
\end{equation}

When consider the conformal metric $\tg=e^{2u}g$ with 
\begin{equation}
	e^u=\frac{2}{1+\sum_{i=1}^n (x^i)^2+(x^{n+1})^2}=\frac{2}{2+(1-1/a^2)f^2},
\end{equation}
we have
\begin{equation}
	u_{n+1}:=e_{n+1} u = e^u\frac{1}{w}\Big(\frac{a^2\sum_{i=1}^n(x^i)^2}{f}+x^{n+1}\Big)=e^u\frac{a^2}{wf}.
\end{equation}

Hence, by \eqref{eq-208} and above equations,
the squared norm of the second fundamental form w.r.t the conformal metric is
\begin{equation}
		|\th|_{\tg}^2= e^{-2u}\big(|h|^2+n (u_{n+1})^2-2 n u_{n+1}\big)=\frac{(a^2-1)^2}{4}G_n((x^{n+1})^2),
\end{equation}
where
\begin{equation}
	\begin{aligned}
		G_n(t)=\frac{ (n-1)(a^2-1)^2 t^4 -2 (n-1) a^4 (a^2-1) t^3 + a^4 ((n-1) a^4 +9) t^2 - 12 a^6 t +4 a^8 }{(a^4+(1-a^2)t)^3}.
	\end{aligned}
\end{equation}
Now let $n=4$.
One can check that
	\begin{equation}
		\max_{0\le t\le a^2}G_4(t)=\begin{cases}
			4/a^4, & \mbox{for $0<a\le 1$};\\
			4a^2, & \mbox{for $a>1$}.
		\end{cases}
	\end{equation}
	Hence,
	\begin{equation}
		\max|\th|^2_{\tg}=\begin{cases}
			(1/a^2-1)^2, & \mbox{for $0<a\le 1$};\\
			(a^2-1)^2 a^2, & \mbox{for $a>1$}.
		\end{cases}
	\end{equation}

\end{document}